\RequirePackage{fix-cm}
\documentclass[smallextended]{svjour3}       
\smartqed  
\usepackage{graphicx}
\usepackage{geometry}                
\usepackage{graphicx}
\usepackage{amssymb}
\usepackage{epstopdf}
\usepackage{float}

\usepackage{amsmath,amsfonts,amsthm,url,array,etoolbox,subcaption}
\usepackage{enumerate}
\usepackage{tikz,pgfplots}
\usetikzlibrary{calc,decorations.markings}
\usetikzlibrary{shapes,snakes}
\theoremstyle{plain}
\newtheorem{thm}{Theorem}

\newtheorem{prop}[thm]{Proposition}

\theoremstyle{definition}

\theoremstyle{remark}

\begin{document}

\title{Lattice Surfaces and smallest triangles
}


\author{Chenxi Wu        
}


\institute{Chenxi Wu \at
              Department of Mathematics, Cornell University \\
              Tel.: +1-607-2277644\\
              \email{cw538@cornell.edu}           
}

\date{Received: date / Accepted: date}

\maketitle

\begin{abstract}
We calculate the area of the smallest triangle and the area of the smallest virtual triangle for many known lattice surfaces. We show that our list of the lattice surfaces for which the area of the smallest virtual triangle greater than $1\over 20$ is complete. In particular, this means that there are no new lattice surfaces for which the area of the smallest virtual triangle is greater than .05. Our method follows an
algorithm described by Smillie and Weiss and improves on it in certain respects.
\end{abstract}

\section{Introduction}
\label{intro}
The Veech group of a translation surface is the discrete subgroup of $SL(2,\mathbb{R})$ consisting of the derivatives of its affine automorphisms. 
Lattice surfaces, or Veech surfaces, are translation surfaces whose Veech groups are lattices. 
They can be seen as the generalization of the flat torus, and their geometric and dynamical properties have been extensively studied. 
For example, Veech \cite{veech1989teichmuller} showed that the growth rate of the holonomies of saddle connections in any Veech surface must be asymptotically quadratic. 
He also proved the Veech dichotomy, which says that the translation flow in any given Veech surface must be either minimal or completely periodic.\\

The flat torus is a lattice surface, so are its branched covers branching at one point. 
These are called square-tiled, or arithmetic, lattice surfaces. 
For square-tiled lattice surfaces, Delecroix implemented algorithms on Veech group, orbit graph and Lyapunov exponents in Sage, and listed square-tiled lattice surfaces that can be built from up to 10 squares. 
The first class of non-squared tiled lattice surfaces was found by Veech \cite{veech1989teichmuller}.
The following is a list of some known lattice surfaces (c.f.\cite{wright2014translation}):
\begin{enumerate}[(1)]
	\item The flat torus.
	\item Eigenforms in $\mathcal{H}(2)$ \cite{1,2}.
	\item Surfaces in the Prym eigenform loci of genus 3 and 4 \cite{3}.
	\item The Bouw-M\"oller family \cite{bm}, which includes the examples found by Veech and
	Ward.
	\item Isolated examples found by Vorobets \cite{4}, Kenyon-Smillie \cite{kenyon2000billiards}.
	\item Lattice surfaces formed by covering construction of the above.
\end{enumerate}

There are two new infinite families of lattice surfaces in the forthcoming works of McMullen-Mukamel-Wright and Eskin-McMullen-Mukamel-Wright.\\

Smillie and Weiss \cite{smillie2004minimal} showed that lattice surfaces are the translation surfaces whose $GL(2,\mathbb{R})$ orbits in their strata are closed. 
These closed $GL(2,\mathbb{R})$-orbits are called Teichm\"uller curves. 
According to \cite{eskin2013isolation}, all $GL(2,\mathbb{R})$-orbit closures in the strata are affine submanifolds.
Teichm\"uller curves are the orbit closures of the lowest possible dimension.\\

McMullen \cite{1} classified all the genus 2 lattice surfaces.
Kenyon-Smillie \cite{kenyon2000billiards}, Bainbridge-M\"oller \cite{bainbridge2012deligne}, Bainbridge-Habegger-M\"oller \cite{bainbridge2014teichmueller}, Matheus-Wright \cite{matheus2015hodge}, Lanneau-Nguyen-Wright \cite{lanneau2015finiteness} and others established finiteness results of Teichm\"uller curves in different settings. The forthcoming work of Eskin-Filip-Wright established finiteness result that are in a sense optimal.\\

Besides studying the finiteness of Teichm\"uller curves in given strata, another way to study lattice surfaces is to enumerate them through the minimal area of an embedded triangle formed by saddle connections, which characterizes the complexity of the surface in some sense. 
For example, a natural way to characterize the complexity of a square-tiled surface is the number of squares needed to construct this 
surface. If a surface $M$ is tiled by $N$ squares, the minimal area of an embedded triangle formed by saddle connections must be at least ${1\over 2N}$ of the total area. 
Vorobets \cite{4}, Smillie and Weiss \cite{5} established the existence of a lower bound of the area of such triangle for lattice surfaces. 
Hence, for a lattice surface $M$, we define the minimal area of triangle $Area_T(M)=\inf\{Area_{\Delta}\}/Area(M)$, where $\inf\{Area_{\Delta}\}$ is the minimal area of embedded triangles formed by saddle connections, and $Area(M)$ is the total area of the surface.\\

Furthermore, Smillie and Weiss \cite{5} showed that given $\epsilon>0$, any flat surface with $Area_T>\epsilon$ or lies on one of finitely many Teichm\"uller curves. 
From this, they developed an algorithm to find all lattice surfaces and list them in order of complexity.
A related notion introduced in \cite{5} is the minimal area of a virtual triangle, defined as $Area_{VT}(M)={1\over 2}\inf_{l,l'}{||l\times l'||}/Area(M)$, where $l$ and $l'$ are holonomies of non-parallel saddle connections, and $Area(M)$ is the total area of the surface. Smillie and Weiss \cite{5} gave the first lattice surfaces obtained by their algorithm, Samuel Leli\'evre found further examples and showed that it is interesting to plot $Area_T$ against $Area_{VT}$. 
Yumin Zhong \cite{yz} showed that the double regular pentagon has the smallest $Area_T$ among non-arithmetic lattice surfaces.\\

In this paper, we calculate the quantities $Area_T$ and $Area_{VT}$ of all published primitive lattice surfaces except those in the Prym eigenform loci of $\mathcal{H}(6)$. 
We also provide a list of the Veech surfaces with $Area_{VT}>0.05$.\\

In section 3, we calculate all the lattice surfaces with $Area_{VT}>0.05$ using the method outlined in \cite{5} with some improvements which we describe there. 
This method will eventually produce all lattice surfaces in theory, but the amount of computation may grow very fast as the bound on $Area_{VT}$ decreases. 
We show that the published list of lattice surfaces is complete up to $Area_{VT}>0.05$. 
More specifically, we show the following:
\begin{thm}The following is a complete list of lattice surfaces for which $Area_{VT}(M)>0.05$:\begin{enumerate}[(1)]
		\item Square-tiled surfaces having less than $10$ squares.
		\item Lattice surfaces in $\mathcal{H}(2)$ with discriminant $5$, $8$ or $17$.
		\item Lattice surfaces in the Prym eigenform loci in $\mathcal{H}(4)$ with discriminant $8$.
	\end{enumerate}
\end{thm}

There are two Teichm\"uller curves in $\mathcal{H}(2)$ that have discriminant 17, so there are five different non-arithmetic lattice surfaces up to affine action that have $Area_{VT}>0.05$. \\

Together with the calculation on square-tiled surfaces done by Delecroix, the numbers of surfaces for each given $Area_{VT}>0.05$ are calculated and summarized in Table~\ref{table1}.\\

\begin{table} 
\caption{Lattice surfaces with $Area_{VT}>0.5$}
\label{table1}	
\begin{tabular}{c|c|c}
	\hline\noalign{\smallskip}
	$Area_{VT}$&num. of surfaces&\\
	\noalign{\smallskip}\hline\noalign{\smallskip}
	1/2&1&flat torus\\
	1/6&1&square tiled\\
	1/8&3&square tiled\\
	1/10&7&square tiled\\
	0.0854102&1&double regular pentagon\\
	1/12&25&square tiled\\
	0.0732233&1&regular octagon\\
	1/14&40&square tiled\\
	1/16&113&square tiled\\
	1/18&125&square tiled\\
	0.0531695&2&genus $2$ lattice surface with discriminant $17$\\
	0.0517767&1&Prym surface in genus 3 with discriminant 8,\\
	&&which is also the Bouw-M\"oller surface $BM(3,4)$\\
	\noalign{\smallskip}\hline
\end{tabular}
\end{table}

In Section 2, we calculate the $Area_T$ and $Area_{VT}$ of lattice surfaces in $\mathcal{H}(2)$, in the Prym loci of $\mathcal{H}(4)$, in the Bouw-M\"oller family, as well as in the isolated Teichm\"uller curves discovered by Vorobets \cite{4} and Kenyon-Smillie \cite{kenyon2000billiards}.
Figures~\ref{p0}-\ref{p0d} show the $Area_T$ and $Area_{VT}$ of these lattice surfaces.\\ 

\begin{figure}[h]
	\begin{tikzpicture}
	\begin{axis}[
	axis lines=center,
	xlabel=$Area_T$, ylabel=$Area_{VT}$,   xtick={0,0.02,0.04,0.06,0.08,0.1,0.12},    xticklabels={$0$,$0.02$, $0.04$, $0.06$, $0.08$,$0.1$,$0.12$},
	scale=1.6] 
	\addplot[color=red,only marks,mark=x,mark size=1] coordinates { 
		(0.138197,0.0854102)
		(0.0419874,0.0419874)
		(0.0680983,0.0531695)
		(0.0863366,0.0227724)
		(0.0178808,0.0178808)
		(0.0324029,0.0324029)
		(0.0445013,0.0376982)
		(0.0547828,0.0233041)
		(0.063661,0.0108746)
		(0.00961901,0.00961901)
		(0.0182068,0.0182068)
		(0.0259355,0.0259355)
		(0.0329392,0.0290782)
		(0.0393248,0.0208682)
		(0.045178,0.0133426)
		(0.050569,0.0064113)
		(0.00595323,0.00595323)
		(0.0115005,0.0115005)
		(0.0166859,0.0166859)
		(0.0215471,0.0215471)
		(0.0261166,0.0236352)
		(0.0304225,0.0183725)
		(0.0344891,0.0134022)
		(0.0383378,0.00869824)
		(0.0419874,0.00423759)
		(0.00403252,0.00403252)
		(0.007876,0.007876)
		(0.0115448,0.0115448)
		(0.0150516,0.0150516)
		(0.0184082,0.0184082)
		(0.0216249,0.0198978)
		(0.0247112,0.0162504)
		(0.0276757,0.0127469)
		(0.0305261,0.00937824)
		(0.0332696,0.00613588)
		(0.0359128,0.00301216)
		(0.00290708,0.00290708)
		(0.00571504,0.00571504)
		(0.0084294,0.0084294)
		(0.0110552,0.0110552)
		(0.0135973,0.0135973)
		(0.0160599,0.0160599)
		(0.0184471,0.0171764)
		(0.0207627,0.0145046)
		(0.0230102,0.0119114)
		(0.0251928,0.00939292)
		(0.0273137,0.00694577)
		(0.0293756,0.00456659)
		(0.0313813,0.0022523)
		(0.00219302,0.00219302)
		(0.00432933,0.00432933)
		(0.00641132,0.00641132)
		(0.00844127,0.00844127)
		(0.0104213,0.0104213)
		(0.0123534,0.0123534)
		(0.0142395,0.0142395)
		(0.0160814,0.0151077)
		(0.0178808,0.0130684)
		(0.0196393,0.0110754)
		(0.0213585,0.00912707)
		(0.0230397,0.0072217)
		(0.0246843,0.00535775)
		(0.0262938,0.00353374)
		(0.0278692,0.00174827)
		(0.00171235,0.00171235)
		(0.00338999,0.00338999)
		(0.00503407,0.00503407)
		(0.00664571,0.00664571)
		(0.00822595,0.00822595)
		(0.0097758,0.0097758)
		(0.0112962,0.0112962)
		(0.0127881,0.0127881)
		(0.0142524,0.0134826)
		(0.0156899,0.011876)
		(0.0171014,0.0102984)
		(0.0184877,0.00874899)
		(0.0198496,0.00722692)
		(0.0211877,0.0057314)
		(0.0225027,0.00426167)
		(0.0237953,0.00281699)
		(0.0250661,0.00139666)
		(0.00137364,0.00137364)
		(0.00272488,0.00272488)
		(0.00405432,0.00405432)
		(0.00536255,0.00536255)
		(0.00665013,0.00665013)
		(0.00791758,0.00791758)
		(0.00916544,0.00916544)
		(0.0103942,0.0103942)
		(0.0116043,0.0116043)
		(0.0127963,0.0121725)
		(0.0139706,0.0108746)
		(0.0151276,0.00959579)
		(0.0162678,0.00833561)
		(0.0173915,0.00709361)
		(0.0184992,0.00586934)
		(0.0195912,0.00466241)
		(0.0206678,0.00347239)
		(0.0217296,0.0022989)
		(0.0227767,0.00114156)
		(0.00112613,0.00112613)
		(0.00223718,0.00223718)
		(0.00333348,0.00333348)
		(0.00441536,0.00441536)
		(0.00548312,0.00548312)
		(0.00653708,0.00653708)
		(0.00757752,0.00757752)
		(0.00860474,0.00860474)
		(0.00961901,0.00961901)
		(0.0106206,0.0106206)
		(0.0116098,0.0110941)
		(0.0125868,0.010024)
		(0.0135519,0.00896698)
		(0.0145053,0.00792275)
		(0.0732233,0.0732233)
		(0.105662,0.0386751)
		(0.0263932,0.0263932)
		(0.0458759,0.0458759)
		(0.0610178,0.0334734)
		(0.0732233,0.015165)
		(0.0128292,0.0128292)
		(0.0238665,0.0238665)
		(0.0334936,0.0334936)
		(0.0419874,0.0273501)
		(0.0495541,0.0172612)
		(0.0563508,0.00819889)
		(0.00746437,0.00746437)
		(0.0142977,0.0142977)
		(0.0205843,0.0205843)
		(0.0263932,0.0263932)
		(0.0317821,0.0227724)
		(0.0367993,0.0165009)
		(0.0414856,0.010643)
		(0.0458759,0.00515518)
		(0.00485483,0.00485483)
		(0.00943739,0.00943739)
		(0.0137722,0.0137722)
		(0.0178808,0.0178808)
		(0.0217823,0.0217823)
		(0.0254934,0.019408)
		(0.0290291,0.015165)
		(0.0324029,0.0111165)
		(0.0356268,0.00724788)
		(0.0387114,0.00354628)
		(0.00340152,0.00340152)
		(0.00666787,0.00666787)
		(0.00980777,0.00980777)
		(0.0128292,0.0128292)
		(0.0157394,0.0157394)
		(0.018545,0.018545)
		(0.0212521,0.0168725)
		(0.0238665,0.0138224)
		(0.0263932,0.0108746)
		(0.0288371,0.00802342)
		(0.0312025,0.00526374)
		(0.0334936,0.00259074)
		(0.00251263,0.00251263)
		(0.00495099,0.00495099)
		(0.00731866,0.00731866)
		(0.00961901,0.00961901)
		(0.0118552,0.0118552)
		(0.01403,0.01403)
		(0.0161464,0.0161464)
		(0.0182068,0.0149065)
		(0.0202137,0.0126129)
		(0.0221694,0.0103778)
		(0.024076,0.00819889)
		(0.0259355,0.00607376)
		(0.0277498,0.00400025)
		(0.0295207,0.00197632)
		(0.00193053,0.00193053)
		(0.00381702,0.00381702)
		(0.00566111,0.00566111)
		(0.00746437,0.00746437)
		(0.00922829,0.00922829)
		(0.0109543,0.0109543)
		(0.0126437,0.0126437)
		(0.0142977,0.0142977)
		(0.0159177,0.0133426)
		(0.0175047,0.0115572)
		(0.0190599,0.00980762)
		(0.0205843,0.0080927)
		(0.0220788,0.0064113)
		(0.0235446,0.00476233)
		(0.0249824,0.00314478)
		(0.0263932,0.00155765)
		(0.00152907,0.00152907)
		(0.00303042,0.00303042)
		(0.00450487,0.00450487)
		(0.00595323,0.00595323)
		(0.00737626,0.00737626)
		(0.00877468,0.00877468)
		(0.0101492,0.0101492)
		(0.0115005,0.0115005)
		(0.0128292,0.0128292)
		(0.0141359,0.0120712)
		(0.0154213,0.010643)
		(0.0166859,0.00923792)
		(0.0179302,0.00785531)
		(0.0191549,0.00649459)
		(0.0203603,0.00515518)
		(0.0215471,0.00383654)
		(0.0227157,0.00253814)
		(0.0238665,0.00125945)
		(0.0012407,0.0012407)
		(0.00246311,0.00246311)
		(0.00366768,0.00366768)
		(0.00485483,0.00485483)
		(0.00602498,0.00602498)
		(0.00717853,0.00717853)
		(0.00831588,0.00831588)
		(0.00943739,0.00943739)
		(0.0105434,0.0105434)
		(0.0116344,0.0116344)
		(0.0127105,0.0110184)
		(0.0137722,0.00985058)
		(0.0148198,0.00869824)
		(0.0158535,0.0075611)
		(0.0168738,0.00643882)
		(0.0178808,0.00533109)
		(0.0188749,0.00423759)
		(0.0198563,0.00315802)
		(0.0208254,0.00209209)
		(0.0217823,0.00103951)
		(0.0010267,0.0010267)
		(0.00204085,0.00204085)
		(0.00304271,0.00304271)
		(0.00403252,0.00403252)
	}; 
	\addplot[color=green,only marks,mark=+,mark size=1] coordinates{
		(0.0517767,0.0517767)
		(0.0528312,0.0386751)
		(0.0263932,0.0263932)
		(0.0229379,0.0229379)
		(0.0167367,0.0167367)
		(0.0366117,0.015165)
		(0.0128292,0.0128292)
		(0.0119332,0.0119332)
		(0.0193376,0.0193376)
		(0.0209937,0.0209937)
		(0.00863062,0.00863062)
		(0.0281754,0.00819889)
		(0.00746437,0.00746437)
		(0.00714887,0.00714887)
		(0.0183848,0.0183848)
		(0.0131966,0.0131966)
		(0.0113862,0.0113862)
		(0.0183996,0.0165009)
		(0.00532151,0.00532151)
		(0.0229379,0.00515518)
		(0.00485483,0.00485483)
		(0.00471869,0.00471869)
		(0.0137722,0.0137722)
		(0.00894041,0.00894041)
		(0.0119306,0.0119306)
		(0.0127467,0.0127467)
		(0.00758252,0.00758252)
		(0.0162015,0.0111165)
		(0.00362394,0.00362394)
		(0.0193557,0.00354628)
		(0.00340152,0.00340152)
		(0.00333393,0.00333393)
		(0.00980777,0.00980777)
		(0.00641459,0.00641459)
		(0.011652,0.011652)
		(0.00927249,0.00927249)
		(0.00843625,0.00843625)
		(0.0119332,0.0119332)
		(0.0054373,0.0054373)
		(0.0144185,0.00802342)
		(0.00263187,0.00263187)
		(0.0167468,0.00259074)
		(0.00251263,0.00251263)
		(0.00247549,0.00247549)
		(0.00731866,0.00731866)
		(0.00480951,0.00480951)
		(0.0110828,0.0110828)
		(0.00701502,0.00701502)
		(0.00863062,0.00863062)
		(0.00910342,0.00910342)
		(0.00630643,0.00630643)
		(0.0110847,0.0103778)
		(0.00409944,0.00409944)
		(0.0129677,0.00607376)
		(0.00200013,0.00200013)
		(0.0147604,0.00197632)
		(0.00193053,0.00193053)
		(0.00190851,0.00190851)
		(0.00566111,0.00566111)
		(0.00373219,0.00373219)
		(0.00922829,0.00922829)
		(0.00547714,0.00547714)
		(0.00851294,0.00851294)
		(0.00714887,0.00714887)
		(0.00667129,0.00667129)
		(0.00875236,0.00875236)
		(0.00490381,0.00490381)
		(0.0102921,0.0080927)
		(0.00320565,0.00320565)
		(0.0117723,0.00476233)
		(0.00157239,0.00157239)
		(0.0131966,0.00155765)
		(0.00152907,0.00152907)
		(0.00151521,0.00151521)
		(0.00450487,0.00450487)
		(0.00297662,0.00297662)
		(0.00737626,0.00737626)
		(0.00438734,0.00438734)
		(0.00825045,0.00825045)
		(0.00575024,0.00575024)
		(0.00676157,0.00676157)
		(0.00706796,0.00706796)
		(0.00532151,0.00532151)
		(0.00834293,0.00834293)
		(0.00392766,0.00392766)
		(0.00957744,0.00649459)
		(0.00257759,0.00257759)
		(0.0107736,0.00383654)
		(0.00126907,0.00126907)
		(0.0119332,0.00125945)
		(0.0012407,0.0012407)
		(0.00123156,0.00123156)
		(0.00366768,0.00366768)
		(0.00242742,0.00242742)
		(0.00602498,0.00602498)
		(0.00358927,0.00358927)
		(0.00792627,0.00792627)
		(0.00471869,0.00471869)
		(0.00670111,0.00670111)
		(0.00581718,0.00581718)
		(0.00550922,0.00550922)
		(0.0068861,0.0068861)
		(0.00434912,0.00434912)
		(0.00792677,0.0075611)
		(0.00321941,0.00321941)
		(0.00894041,0.00533109)
		(0.00211879,0.00211879)
		(0.00992817,0.00315802)
		(0.00104604,0.00104604)
		(0.0108911,0.00103951)
		(0.0010267,0.0010267)
		(0.00102043,0.00102043)
		(0.00304271,0.00304271)
		(0.00201626,0.00201626)
		(0.0340491,0.0340491)
		(0.0162015,0.0162015)
		(0.0273914,0.0233041)
		(0.00910342,0.00910342)
		(0.0164696,0.0164696)
		(0.022589,0.0133426)
		(0.00575024,0.00575024)
		(0.0107736,0.0107736)
		(0.0152112,0.0152112)
		(0.0191689,0.00869824)
		(0.003938,0.003938)
		(0.00752582,0.00752582)
		(0.0108125,0.0108125)
		(0.0138378,0.0127469)
		(0.0166348,0.00613588)
		(0.00285752,0.00285752)
		(0.00552762,0.00552762)
		(0.00802994,0.00802994)
		(0.0103813,0.0103813)
		(0.0125964,0.00939292)
		(0.0146878,0.00456659)
		(0.00216467,0.00216467)
		(0.00422063,0.00422063)
		(0.0061767,0.0061767)
		(0.00804071,0.00804071)
		(0.00981967,0.00981967)
		(0.0115198,0.0072217)
		(0.0131469,0.00353374)
		(0.00169499,0.00169499)
		(0.00332285,0.00332285)
		(0.0048879,0.0048879)
		(0.00639407,0.00639407)
		(0.00784496,0.00784496)
		(0.00924387,0.00874899)
		(0.0105938,0.0057314)
		(0.0118977,0.00281699)
		(0.00136244,0.00136244)
		(0.00268128,0.00268128)
		(0.00395879,0.00395879)
		(0.0051971,0.0051971)
		(0.00639816,0.00639816)
		(0.00756381,0.00756381)
		(0.00869575,0.00709361)
		(0.00979558,0.00466241)
		(0.0108648,0.0022989)
		(0.00111859,0.00111859)
		(0.00220768,0.00220768)
		(0.00326854,0.00326854)
		(0.00430237,0.00430237)
		(0.0053103,0.0053103)
		(0.0062934,0.0062934)
		(0.00725266,0.00725266)
	};
	\addplot[color=blue,only marks,mark=o,mark size=1] coordinates{
		(0.0732233,0.0732233)
		(0.0381966,0.0381966)
		(0.138197,0.0854102)
		(0.0537872,0.0298496)
		(0.0259951,0.0138317)
		(0.0517767,0.0517767)
		(0.0263932,0.0263932)
		(0.0161134,0.0113939)
		(0.0152511,0.0152511)
		(0.0374575,0.0374575)
		(0.00556222,0.00343764)
		(0.00959821,0.00959821)
		(0.00591286,0.00418102)
		(0.00351675,0.00351675)
		(0.00219681,0.00126833)
		(0.00642815,0.00642815)
		(0.0158728,0.0158728)
		(0.00236212,0.00145987)
		(0.00590471,0.00590471)
		(0.000972257,0.000539562)
		(0.004514,0.004514)
		(0.00279046,0.00197315)
		(0.00166199,0.00166199)
		(0.00103894,0.000599832)
		(0.000684393,0.000684393)
		(0.000471823,0.000255349)
		(0.00329041,0.00329041)
		(0.00814427,0.00814427)
		(0.00121316,0.000749776)
		(0.0030341,0.0030341)
		(0.000499731,0.00027733)
		(0.00137817,0.00137817)
		(0.000246739,0.000131287)
	}; 
	\addplot[color=black,only marks,mark=*,mark size=1] coordinates{
		(0.0528312,0.0386751)
		(0.0259951,0.0169671)
		(0.014189,0.014189)
	};
	\end{axis} 
	\end{tikzpicture}
	\caption{\label{p0} Red, green, blue and black points correspond to non-square-tiled lattice surfaces described in Theorem 1.2, 1.3, 1.4 and 1.5 respectively.}
\end{figure}
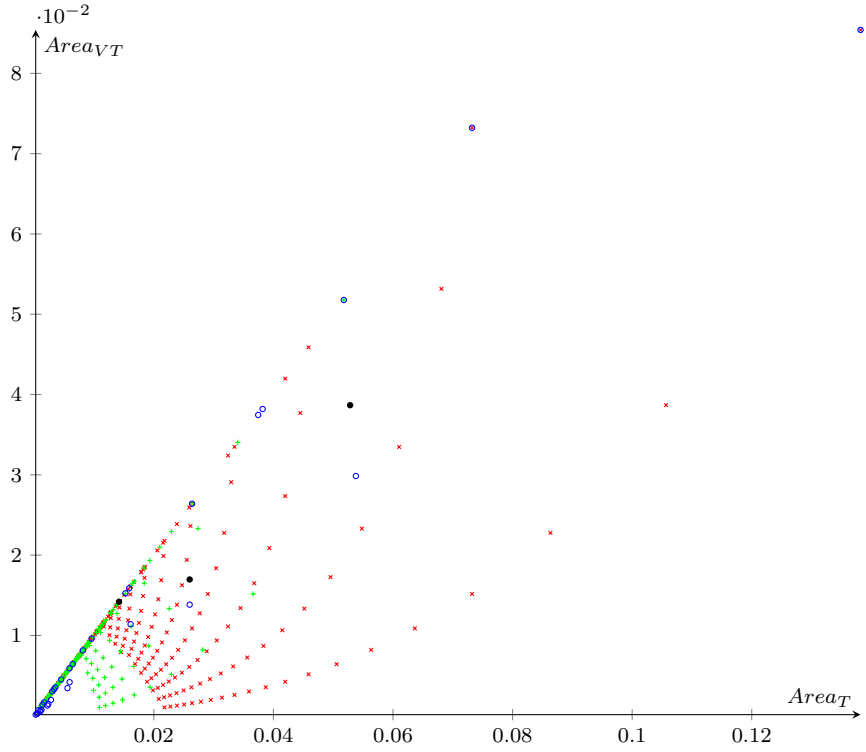

\begin{figure}
	\includegraphics[scale=0.4]{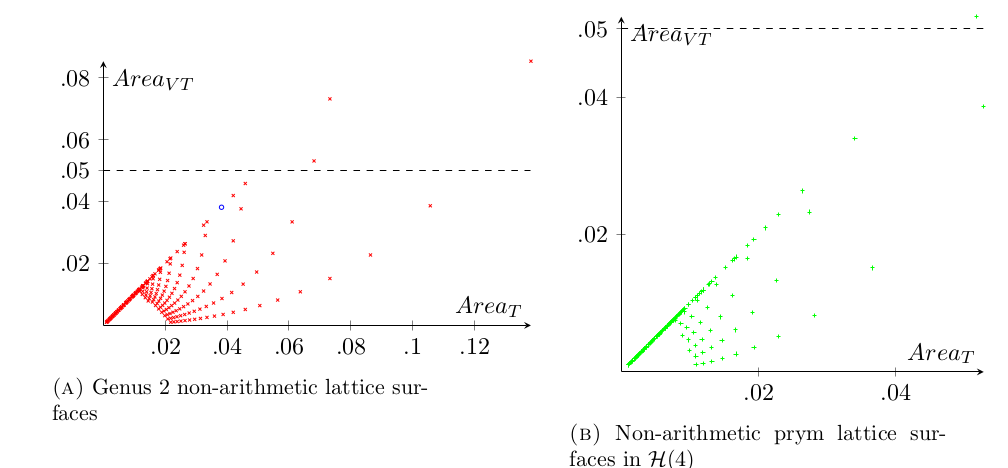}
	\caption{\label{p0b}}
\end{figure}

\begin{figure}
	\includegraphics[scale=0.4]{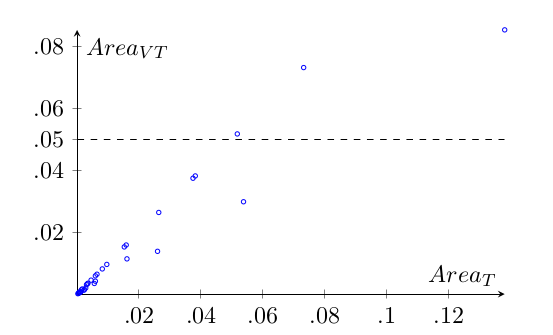}
	\caption{\label{p0c}Non-arithmetic cases in the Bouw-M\"oller family}
\end{figure}

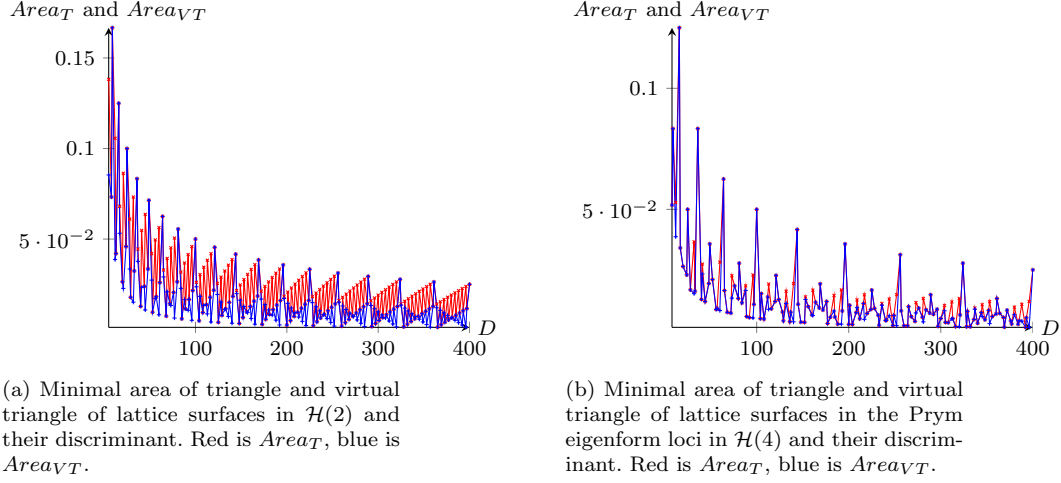
\begin{figure}
	\begin{subfigure}[t]{0.35\textwidth}
		\begin{tikzpicture}
		\begin{axis}[
		axis lines=center,
		xlabel=$D$, ylabel=$Area_T\text{ and }Area_{VT}$, scale=0.7,
		every axis x label/.style={
			at={(current axis.right of origin)},
			anchor=west,
		},
		every axis y label/.style={
			at={(current axis.above origin)},
			anchor=south,
		}] 
		\addplot[color=red,mark=x,mark size=0.8] coordinates {
			(5,0.138197)
			(8,0.0732233)
			(9,0.166667)
			(12,0.105662)
			(13,0.0419874)
			(16,0.125)
			(17,0.0680983)
			(20,0.0263932)
			(21,0.0863366)
			(24,0.0458759)
			(25,0.1)
			(28,0.0610178)
			(29,0.0178808)
			(32,0.0732233)
			(33,0.0324029)
			(36,0.0833333)
			(37,0.0445013)
			(40,0.0128292)
			(41,0.0547828)
			(44,0.0238665)
			(45,0.063661)
			(48,0.0334936)
			(49,0.0714286)
			(52,0.0419874)
			(53,0.00961901)
			(56,0.0495541)
			(57,0.0182068)
			(60,0.0563508)
			(61,0.0259355)
			(64,0.0625)
			(65,0.0329392)
			(68,0.00746437)
			(69,0.0393248)
			(72,0.0142977)
			(73,0.045178)
			(76,0.0205843)
			(77,0.050569)
			(80,0.0263932)
			(81,0.0555556)
			(84,0.0317821)
			(85,0.00595323)
			(88,0.0367993)
			(89,0.0115005)
			(92,0.0414856)
			(93,0.0166859)
			(96,0.0458759)
			(97,0.0215471)
			(100,0.05)
			(101,0.0261166)
			(104,0.00485483)
			(105,0.0304225)
			(108,0.00943739)
			(109,0.0344891)
			(112,0.0137722)
			(113,0.0383378)
			(116,0.0178808)
			(117,0.0419874)
			(120,0.0217823)
			(121,0.0454545)
			(124,0.0254934)
			(125,0.00403252)
			(128,0.0290291)
			(129,0.007876)
			(132,0.0324029)
			(133,0.0115448)
			(136,0.0356268)
			(137,0.0150516)
			(140,0.0387114)
			(141,0.0184082)
			(144,0.0416667)
			(145,0.0216249)
			(148,0.00340152)
			(149,0.0247112)
			(152,0.00666787)
			(153,0.0276757)
			(156,0.00980777)
			(157,0.0305261)
			(160,0.0128292)
			(161,0.0332696)
			(164,0.0157394)
			(165,0.0359128)
			(168,0.018545)
			(169,0.0384615)
			(172,0.0212521)
			(173,0.00290708)
			(176,0.0238665)
			(177,0.00571504)
			(180,0.0263932)
			(181,0.0084294)
			(184,0.0288371)
			(185,0.0110552)
			(188,0.0312025)
			(189,0.0135973)
			(192,0.0334936)
			(193,0.0160599)
			(196,0.0357143)
			(197,0.0184471)
			(200,0.00251263)
			(201,0.0207627)
			(204,0.00495099)
			(205,0.0230102)
			(208,0.00731866)
			(209,0.0251928)
			(212,0.00961901)
			(213,0.0273137)
			(216,0.0118552)
			(217,0.0293756)
			(220,0.01403)
			(221,0.0313813)
			(224,0.0161464)
			(225,0.0333333)
			(228,0.0182068)
			(229,0.00219302)
			(232,0.0202137)
			(233,0.00432933)
			(236,0.0221694)
			(237,0.00641132)
			(240,0.024076)
			(241,0.00844127)
			(244,0.0259355)
			(245,0.0104213)
			(248,0.0277498)
			(249,0.0123534)
			(252,0.0295207)
			(253,0.0142395)
			(256,0.03125)
			(257,0.0160814)
			(260,0.00193053)
			(261,0.0178808)
			(264,0.00381702)
			(265,0.0196393)
			(268,0.00566111)
			(269,0.0213585)
			(272,0.00746437)
			(273,0.0230397)
			(276,0.00922829)
			(277,0.0246843)
			(280,0.0109543)
			(281,0.0262938)
			(284,0.0126437)
			(285,0.0278692)
			(288,0.0142977)
			(289,0.0294118)
			(292,0.0159177)
			(293,0.00171235)
			(296,0.0175047)
			(297,0.00338999)
			(300,0.0190599)
			(301,0.00503407)
			(304,0.0205843)
			(305,0.00664571)
			(308,0.0220788)
			(309,0.00822595)
			(312,0.0235446)
			(313,0.0097758)
			(316,0.0249824)
			(317,0.0112962)
			(320,0.0263932)
			(321,0.0127881)
			(324,0.0277778)
			(325,0.0142524)
			(328,0.00152907)
			(329,0.0156899)
			(332,0.00303042)
			(333,0.0171014)
			(336,0.00450487)
			(337,0.0184877)
			(340,0.00595323)
			(341,0.0198496)
			(344,0.00737626)
			(345,0.0211877)
			(348,0.00877468)
			(349,0.0225027)
			(352,0.0101492)
			(353,0.0237953)
			(356,0.0115005)
			(357,0.0250661)
			(360,0.0128292)
			(361,0.0263158)
			(364,0.0141359)
			(365,0.00137364)
			(368,0.0154213)
			(369,0.00272488)
			(372,0.0166859)
			(373,0.00405432)
			(376,0.0179302)
			(377,0.00536255)
			(380,0.0191549)
			(381,0.00665013)
			(384,0.0203603)
			(385,0.00791758)
			(388,0.0215471)
			(389,0.00916544)
			(392,0.0227157)
			(393,0.0103942)
			(396,0.0238665)
			(397,0.0116043)
			(400,0.025)	}; 
		\addplot[color=blue,mark=+,mark size=0.8] coordinates {
			(5,0.0854102)
			(8,0.0732233)
			(9,0.166667)
			(12,0.0386751)
			(13,0.0419874)
			(16,0.125)
			(17,0.0531695)
			(20,0.0263932)
			(21,0.0227724)
			(24,0.0458759)
			(25,0.1)
			(28,0.0334734)
			(29,0.0178808)
			(32,0.015165)
			(33,0.0324029)
			(36,0.0833333)
			(37,0.0376982)
			(40,0.0128292)
			(41,0.0233041)
			(44,0.0238665)
			(45,0.0108746)
			(48,0.0334936)
			(49,0.0714286)
			(52,0.0273501)
			(53,0.00961901)
			(56,0.0172612)
			(57,0.0182068)
			(60,0.00819889)
			(61,0.0259355)
			(64,0.0625)
			(65,0.0290782)
			(68,0.00746437)
			(69,0.0208682)
			(72,0.0142977)
			(73,0.0133426)
			(76,0.0205843)
			(77,0.0064113)
			(80,0.0263932)
			(81,0.0555556)
			(84,0.0227724)
			(85,0.00595323)
			(88,0.0165009)
			(89,0.0115005)
			(92,0.010643)
			(93,0.0166859)
			(96,0.00515518)
			(97,0.0215471)
			(100,0.05)
			(101,0.0236352)
			(104,0.00485483)
			(105,0.0183725)
			(108,0.00943739)
			(109,0.0134022)
			(112,0.0137722)
			(113,0.00869824)
			(116,0.0178808)
			(117,0.00423759)
			(120,0.0217823)
			(121,0.0454545)
			(124,0.019408)
			(125,0.00403252)
			(128,0.015165)
			(129,0.007876)
			(132,0.0111165)
			(133,0.0115448)
			(136,0.00724788)
			(137,0.0150516)
			(140,0.00354628)
			(141,0.0184082)
			(144,0.0416667)
			(145,0.0198978)
			(148,0.00340152)
			(149,0.0162504)
			(152,0.00666787)
			(153,0.0127469)
			(156,0.00980777)
			(157,0.00937824)
			(160,0.0128292)
			(161,0.00613588)
			(164,0.0157394)
			(165,0.00301216)
			(168,0.018545)
			(169,0.0384615)
			(172,0.0168725)
			(173,0.00290708)
			(176,0.0138224)
			(177,0.00571504)
			(180,0.0108746)
			(181,0.0084294)
			(184,0.00802342)
			(185,0.0110552)
			(188,0.00526374)
			(189,0.0135973)
			(192,0.00259074)
			(193,0.0160599)
			(196,0.0357143)
			(197,0.0171764)
			(200,0.00251263)
			(201,0.0145046)
			(204,0.00495099)
			(205,0.0119114)
			(208,0.00731866)
			(209,0.00939292)
			(212,0.00961901)
			(213,0.00694577)
			(216,0.0118552)
			(217,0.00456659)
			(220,0.01403)
			(221,0.0022523)
			(224,0.0161464)
			(225,0.0333333)
			(228,0.0149065)
			(229,0.00219302)
			(232,0.0126129)
			(233,0.00432933)
			(236,0.0103778)
			(237,0.00641132)
			(240,0.00819889)
			(241,0.00844127)
			(244,0.00607376)
			(245,0.0104213)
			(248,0.00400025)
			(249,0.0123534)
			(252,0.00197632)
			(253,0.0142395)
			(256,0.03125)
			(257,0.0151077)
			(260,0.00193053)
			(261,0.0130684)
			(264,0.00381702)
			(265,0.0110754)
			(268,0.00566111)
			(269,0.00912707)
			(272,0.00746437)
			(273,0.0072217)
			(276,0.00922829)
			(277,0.00535775)
			(280,0.0109543)
			(281,0.00353374)
			(284,0.0126437)
			(285,0.00174827)
			(288,0.0142977)
			(289,0.0294118)
			(292,0.0133426)
			(293,0.00171235)
			(296,0.0115572)
			(297,0.00338999)
			(300,0.00980762)
			(301,0.00503407)
			(304,0.0080927)
			(305,0.00664571)
			(308,0.0064113)
			(309,0.00822595)
			(312,0.00476233)
			(313,0.0097758)
			(316,0.00314478)
			(317,0.0112962)
			(320,0.00155765)
			(321,0.0127881)
			(324,0.0277778)
			(325,0.0134826)
			(328,0.00152907)
			(329,0.011876)
			(332,0.00303042)
			(333,0.0102984)
			(336,0.00450487)
			(337,0.00874899)
			(340,0.00595323)
			(341,0.00722692)
			(344,0.00737626)
			(345,0.0057314)
			(348,0.00877468)
			(349,0.00426167)
			(352,0.0101492)
			(353,0.00281699)
			(356,0.0115005)
			(357,0.00139666)
			(360,0.0128292)
			(361,0.0263158)
			(364,0.0120712)
			(365,0.00137364)
			(368,0.010643)
			(369,0.00272488)
			(372,0.00923792)
			(373,0.00405432)
			(376,0.00785531)
			(377,0.00536255)
			(380,0.00649459)
			(381,0.00665013)
			(384,0.00515518)
			(385,0.00791758)
			(388,0.00383654)
			(389,0.00916544)
			(392,0.00253814)
			(393,0.0103942)
			(396,0.00125945)
			(397,0.0116043)
			(400,0.025)
		}; 
		\end{axis} 
		\end{tikzpicture}
		\caption{Minimal area of triangle and virtual triangle of lattice surfaces in $\mathcal{H}(2)$ and their  discriminant. Red is $Area_T$, blue is $Area_{VT}$.}
	\end{subfigure}
	\hspace{2cm}
	\begin{subfigure}[t]{0.35\textwidth}
		\begin{tikzpicture}
		\begin{axis}[
		axis lines=center,
		xlabel=$D$, ylabel=$Area_T\text{ and }Area_{VT}$, scale=0.7,
		every axis x label/.style={
			at={(current axis.right of origin)},
			anchor=west,
		},
		every axis y label/.style={
			at={(current axis.above origin)},
			anchor=south,
		}] 
		\addplot[color=red,mark=x,mark size=0.8] coordinates {
			(8,0.0517767)
			(9,0.0833333)
			(12,0.0528312)
			(16,0.125)
			(17,0.0340491)
			(20,0.0263932)
			(24,0.0229379)
			(25,0.05)
			(28,0.0167367)
			(32,0.0366117)
			(33,0.0162015)
			(36,0.0833333)
			(40,0.0128292)
			(41,0.0273914)
			(44,0.0119332)
			(48,0.0193376)
			(49,0.0357143)
			(52,0.0209937)
			(56,0.00863062)
			(57,0.00910342)
			(60,0.0281754)
			(64,0.0625)
			(65,0.0164696)
			(68,0.00746437)
			(72,0.00714887)
			(73,0.022589)
			(76,0.0183848)
			(80,0.0131966)
			(81,0.0277778)
			(84,0.0113862)
			(88,0.0183996)
			(89,0.00575024)
			(92,0.00532151)
			(96,0.0229379)
			(97,0.0107736)
			(100,0.05)
			(104,0.00485483)
			(105,0.0152112)
			(108,0.00471869)
			(112,0.0137722)
			(113,0.0191689)
			(116,0.00894041)
			(120,0.0119306)
			(121,0.0227273)
			(124,0.0127467)
			(128,0.00758252)
			(129,0.003938)
			(132,0.0162015)
			(136,0.00362394)
			(137,0.00752582)
			(140,0.0193557)
			(144,0.0416667)
			(145,0.0108125)
			(148,0.00340152)
			(152,0.00333393)
			(153,0.0138378)
			(156,0.00980777)
			(160,0.00641459)
			(161,0.0166348)
			(164,0.011652)
			(168,0.00927249)
			(169,0.0192308)
			(172,0.00843625)
			(176,0.0119332)
			(177,0.00285752)
			(180,0.0054373)
			(184,0.0144185)
			(185,0.00552762)
			(188,0.00263187)
			(192,0.0167468)
			(193,0.00802994)
			(196,0.0357143)
			(200,0.00251263)
			(201,0.0103813)
			(204,0.00247549)
			(208,0.00731866)
			(209,0.0125964)
			(212,0.00480951)
			(216,0.0110828)
			(217,0.0146878)
			(220,0.00701502)
			(224,0.00863062)
			(225,0.0166667)
			(228,0.00910342)
			(232,0.00630643)
			(233,0.00216467)
			(236,0.0110847)
			(240,0.00409944)
			(241,0.00422063)
			(244,0.0129677)
			(248,0.00200013)
			(249,0.0061767)
			(252,0.0147604)
			(256,0.03125)
			(257,0.00804071)
			(260,0.00193053)
			(264,0.00190851)
			(265,0.00981967)
			(268,0.00566111)
			(272,0.00373219)
			(273,0.0115198)
			(276,0.00922829)
			(280,0.00547714)
			(281,0.0131469)
			(284,0.00851294)
			(288,0.00714887)
			(289,0.0147059)
			(292,0.00667129)
			(296,0.00875236)
			(297,0.00169499)
			(300,0.00490381)
			(304,0.0102921)
			(305,0.00332285)
			(308,0.00320565)
			(312,0.0117723)
			(313,0.0048879)
			(316,0.00157239)
			(320,0.0131966)
			(321,0.00639407)
			(324,0.0277778)
			(328,0.00152907)
			(329,0.00784496)
			(332,0.00151521)
			(336,0.00450487)
			(337,0.00924387)
			(340,0.00297662)
			(344,0.00737626)
			(345,0.0105938)
			(348,0.00438734)
			(352,0.00825045)
			(353,0.0118977)
			(356,0.00575024)
			(360,0.00676157)
			(361,0.0131579)
			(364,0.00706796)
			(368,0.00532151)
			(369,0.00136244)
			(372,0.00834293)
			(376,0.00392766)
			(377,0.00268128)
			(380,0.00957744)
			(384,0.00257759)
			(385,0.00395879)
			(388,0.0107736)
			(392,0.00126907)
			(393,0.0051971)
			(396,0.0119332)
			(400,0.025)
		}; 
		\addplot[color=blue,mark=+,mark size=0.8] coordinates {
			(8,0.0517767)
			(9,0.0833333)
			(12,0.0386751)
			(16,0.125)
			(17,0.0340491)
			(20,0.0263932)
			(24,0.0229379)
			(25,0.05)
			(28,0.0167367)
			(32,0.015165)
			(33,0.0162015)
			(36,0.0833333)
			(40,0.0128292)
			(41,0.0233041)
			(44,0.0119332)
			(48,0.0193376)
			(49,0.0357143)
			(52,0.0209937)
			(56,0.00863062)
			(57,0.00910342)
			(60,0.00819889)
			(64,0.0625)
			(65,0.0164696)
			(68,0.00746437)
			(72,0.00714887)
			(73,0.0133426)
			(76,0.0183848)
			(80,0.0131966)
			(81,0.0277778)
			(84,0.0113862)
			(88,0.0165009)
			(89,0.00575024)
			(92,0.00532151)
			(96,0.00515518)
			(97,0.0107736)
			(100,0.05)
			(104,0.00485483)
			(105,0.0152112)
			(108,0.00471869)
			(112,0.0137722)
			(113,0.00869824)
			(116,0.00894041)
			(120,0.0119306)
			(121,0.0227273)
			(124,0.0127467)
			(128,0.00758252)
			(129,0.003938)
			(132,0.0111165)
			(136,0.00362394)
			(137,0.00752582)
			(140,0.00354628)
			(144,0.0416667)
			(145,0.0108125)
			(148,0.00340152)
			(152,0.00333393)
			(153,0.0127469)
			(156,0.00980777)
			(160,0.00641459)
			(161,0.00613588)
			(164,0.011652)
			(168,0.00927249)
			(169,0.0192308)
			(172,0.00843625)
			(176,0.0119332)
			(177,0.00285752)
			(180,0.0054373)
			(184,0.00802342)
			(185,0.00552762)
			(188,0.00263187)
			(192,0.00259074)
			(193,0.00802994)
			(196,0.0357143)
			(200,0.00251263)
			(201,0.0103813)
			(204,0.00247549)
			(208,0.00731866)
			(209,0.00939292)
			(212,0.00480951)
			(216,0.0110828)
			(217,0.00456659)
			(220,0.00701502)
			(224,0.00863062)
			(225,0.0166667)
			(228,0.00910342)
			(232,0.00630643)
			(233,0.00216467)
			(236,0.0103778)
			(240,0.00409944)
			(241,0.00422063)
			(244,0.00607376)
			(248,0.00200013)
			(249,0.0061767)
			(252,0.00197632)
			(256,0.03125)
			(257,0.00804071)
			(260,0.00193053)
			(264,0.00190851)
			(265,0.00981967)
			(268,0.00566111)
			(272,0.00373219)
			(273,0.0072217)
			(276,0.00922829)
			(280,0.00547714)
			(281,0.00353374)
			(284,0.00851294)
			(288,0.00714887)
			(289,0.0147059)
			(292,0.00667129)
			(296,0.00875236)
			(297,0.00169499)
			(300,0.00490381)
			(304,0.0080927)
			(305,0.00332285)
			(308,0.00320565)
			(312,0.00476233)
			(313,0.0048879)
			(316,0.00157239)
			(320,0.00155765)
			(321,0.00639407)
			(324,0.0277778)
			(328,0.00152907)
			(329,0.00784496)
			(332,0.00151521)
			(336,0.00450487)
			(337,0.00874899)
			(340,0.00297662)
			(344,0.00737626)
			(345,0.0057314)
			(348,0.00438734)
			(352,0.00825045)
			(353,0.00281699)
			(356,0.00575024)
			(360,0.00676157)
			(361,0.0131579)
			(364,0.00706796)
			(368,0.00532151)
			(369,0.00136244)
			(372,0.00834293)
			(376,0.00392766)
			(377,0.00268128)
			(380,0.00649459)
			(384,0.00257759)
			(385,0.00395879)
			(388,0.00383654)
			(392,0.00126907)
			(393,0.0051971)
			(396,0.00125945)
			(400,0.025)
		}; 
		\end{axis} 
		\end{tikzpicture}
		\caption{Minimal area of triangle and virtual triangle of lattice surfaces in the Prym eigenform loci in $\mathcal{H}(4)$ and their discriminant. Red is $Area_T$, blue is $Area_{VT}$.}
	\end{subfigure}
	
	\caption{\label{p0d}The pattern is related to the fractional part of the sequence$\sqrt{n}$.}
\end{figure}

These results are summarized in the following theorems:
\begin{thm}
	Let $M$ be a lattice surface in $\mathcal{H}(2)$ with discriminant $D$. Then it holds:
	\begin{itemize}
		\item If $D$ is a square, 
		$$Area_T(M)=Area_{VT}(M)={1\over{2\sqrt{D}}}.$$
		\item If $D$ is not a square, let $e_D$ be the largest integer satisfying $e_D\equiv D \mod{2}$ and $e_D<\sqrt{D}$, then 
		$$Area_T(M)= {\sqrt{D}-e_D\over 4\sqrt{D}},\hspace{8pt}Area_{VT}(M)= \min\left({\sqrt{D}-e_D\over 4\sqrt{D}},{2+e_D-\sqrt{D}\over 4\sqrt{D}}\right).$$
	\end{itemize}
\end{thm}

\begin{thm}
	Let $M$ be a lattice surface in the Prym locus in $\mathcal{H}(4)$ with discriminant $D$. Then:
	\begin{itemize}
		\item If $D$ is a square,
		\begin{itemize}
			\item If $D$ is even,
			$$Area_T(M)=Area_{VT}(M)={1\over{2\sqrt{D}}}.$$
			\item If $D$ is odd,
			$$Area_T(M)=Area_{VT}(M)={1\over{4\sqrt{D}}}.$$
		\end{itemize}
		\item If $D$ is not a square, denote by $e'_D$ the largest integer satisfying ${e'_D}^2\equiv D \mod{8}$ and $e'_D<\sqrt{D}$, then:
		\begin{itemize}
			\item If $\sqrt{D}-e'_D<4/3$, 
			$$Area_T(M)=Area_{VT}(M)={{\sqrt{D}-e'_D}\over{8\sqrt{D}}}.$$
			\item If $4/3<\sqrt{D}-e'_D<2$, 
			$$Area_T(M)={{\sqrt{D}-e'_D}\over{8\sqrt{D}}}, Area_{VT}(M)={{2+e'_D-\sqrt{D}}\over{4\sqrt{D}}}.$$
			\item If $2<\sqrt{D}-e'_D<8/3$, 
			$$Area_T(M)=Area_{VT}(M)={{\sqrt{D}-e'_D-2}\over{4\sqrt{D}}}.$$
			\item If $\sqrt{D}-e'_D>8/3$, 
			$$Area_T(M)=Area_{VT}(M)={{4-\sqrt{D}+e'_D}\over{8\sqrt{D}}}.$$
		\end{itemize}
	\end{itemize}
\end{thm}

\begin{thm}The values $Area_T$ and $Area_{VT}$ of lattice surfaces in the Bouw-M\"oller
	family are as follows:
	\begin{itemize}
		\item If $M$ is a regular n-gon where $n$ is even, 
		$$Area_T(M)=Area_{VT}(M)={4\sin\left({\pi\over n}\right)^2\over n}.$$
		\item If $M$ is a double n-gon where $n$ is odd, 
		$$Area_T(M)={2\sin\left({\pi\over n}\right)^2\over n},\hspace{8pt} Area_{VT}(M)={\tan\left({\pi\over n}\right)\sin\left({\pi\over n}\right)\over n}.$$
		\item If $M$ is the Bouw-M\"oller surfaces $S_{m,n}$, $\min(m,n)>2$:\\
		Let 
		$$A=\sum_{k=1}^{n-1}\sin\left({k\pi\over n}\right)^2\cdot \sum_{k=1}^{m-2}\sin\left({k\pi\over m}\right)\sin\left({(k+1)\pi\over m}\right)$$
		$$+\sum_{k=1}^{m-1}\sin\left({k\pi\over m}\right)^2\cdot \sum_{k=1}^{n-2}\sin\left({k\pi\over n}\right)\sin\left({(k+1)\pi\over n}\right).$$
		\begin{itemize}
			\item When $m$ and $n$ are both odd, 
			$$Area_T(M)=Area_{VT}(M)={\sin\left({\pi\over m}\right)^2\sin\left({\pi\over n}\right)^2\cos\left({\pi\over \min(m,n)}\right)\over A}.$$
			\item When $m$ is odd, $n$ is even, or $n$ is odd, $m$ is even, 
			$$Area_T(M)={\sin\left({\pi\over m}\right)^2\sin\left({\pi\over n}\right)^2\cos\left({\pi\over \min(m,n)}\right)\over A},
			\hspace{8pt} Area_{VT}(M)={\sin\left({\pi\over m}\right)^2\sin\left({\pi\over n}\right)^2\over 2A}.$$
			\item When $m$ and $n$ both even,  
			$$Area_T(M)=Area_{VT}={2\sin\left({\pi\over m}\right)^2\sin\left({\pi\over n}\right)^2\cos\left({\pi\over \min(m,n)}\right)\over A}.$$
		\end{itemize}
	\end{itemize}
\end{thm}
The formulas in Theorem 1.4 are derived by the eigenfunctions of grid graphs in \cite{ph}.

\begin{thm}
	The three lattice surfaces in \cite{4} and \cite{kenyon2000billiards} have the following $Area_T$ and $Area_{VT}$:
	\begin{itemize}
		\item The lattice surface obtained from the triangle with angles $(\pi/4,\pi/3,5\pi/12)$ has $Area_T=1/8-\sqrt{3}/24\approx 0.0528312$, $Area_{VT}=\sqrt{3}/6-1/4\approx 0.0386751$.
		\item The lattice surface obtained from the triangle with angles $(2\pi/9,\pi/3,4\pi/9)$ has $Area_T\approx 0.0259951$, $Area_{VT}\approx 0.0169671$.
		\item The lattice surface obtained from the triangle with angles $(\pi/5,\pi/3,7\pi/15)$ has $Area_T=Area_{VT}\approx 0.014189$.
	\end{itemize}
\end{thm}
The values of $Area_T$ and $Area_{VT}$ in Theorem 1.5 are calculated from the eigenvectors corresponding to the leading eigenvalues of graphs $\mathcal{E}_6$, $\mathcal{E}_7$ and $\mathcal{E}_8$.

\section{Calculation of $Area_T$ and $Area_{VT}$}
\label{sec:2}

We begin by proving Theorem 1.2.\\

\begin{proof}[Proof of Theorem 1.2] Veech surfaces in the stratum $\mathcal{H}(2)$ have been described completely by Calta \cite{2} and McMullen \cite{1}. 
	Each of them is associated with an order with a discriminant $D\in \mathbb{Z}$, $D>4$, $D\equiv 0$ or $1 \mod{4}$. 
	There are two Teichm\"uller curves in $\mathcal{H}(2)$ with discriminant $D$ when $D\equiv 1\mod 8$. 
	There is only one Teichm\"uller curve in $\mathcal{H}(2)$ with discriminant $D$ otherwise.\\
	
	When $D$ is a square, the lattice surfaces in $\mathcal{H}(2)$ with discriminant $D$ are square-tiled surfaces, so $Area_T\geq {1\over{2\sqrt{D}}}$, and $Area_{VT}\geq {1\over{2\sqrt{D}}}$. 
	On the other hand, Corollary A2 in \cite{1}, which gives a description of a pair of cylinder decompositions of such surfaces, shows that $Area_T\leq {1\over{2\sqrt{D}}}$, and $Area_{VT}\leq {1\over{2\sqrt{D}}}$. 
	Hence, when $D$ is a square, $Area_T=Area_{VT}={1\over{2\sqrt{D}}}$.\\
	
	Now we consider the case when $D$ is not a square. 
	Consider an embedded triangle on this lattice surface formed by saddle connections. 
	The Veech Dichotomy \cite{veech1989teichmuller} says that the geodesic flow on a lattice surface is either minimal or completely periodic. 
	Hence, any edge of this triangle must lie on a direction where $M$ can be decomposed into 2 cylinders $M=E_1\cup E_2$, as shown in Figure~\ref{p}, where the periodic direction is drawn to be the horizontal direction.\\
	
	\begin{figure}
		\begin{tikzpicture}[ ray/.style={decoration={markings,mark=at position .5 with {
					\arrow[>=latex]{>}}},postaction=decorate}
		]
		\draw[-](0.4,3)--(0,0.5)--(2,0.5)--(2.4,3)--(0.4,3)--(1.5,5)--(4.5,5)--(3.4,3)--(2.4,3);
		\node at (1.2,1.5){$E_1$};
		\node at (2.5,4){$E_2$};
		\node at (1,0.2){$c_1$};
		\node at (3,5.3){$c_2$};
		\draw[->](0.7,0.2)--(0,0.2);
		\draw[->](1.3,0.2)--(2,0.2);
		\draw[-](0,0)--(0,0.4);
		\draw[-](2,0)--(2,0.4);
		\node at (-0.3,1.7){$h_1$};
		\draw[-](-0.5,0.5)--(-0.1,0.5);
		\draw[-](-0.5,3)--(-0.1,3);
		\draw[->](-0.3,2)--(-0.3,3);
		\draw[->](-0.3,1.4)--(-0.3,0.5);
		\draw[->](2.7,5.3)--(1.5,5.3);
		\draw[->](3.3,5.3)--(4.5,5.3);
		\draw[-](1.5,5.1)--(1.5,5.5);
		\draw[-](4.5,5.1)--(4.5,5.5);
		\node at (4.8,4){$h_2$};
		\draw[-](4.6,5)--(5,5);
		\draw[-](4.6,3)--(5,3);
		\draw[->](4.8,4.3)--(4.8,5);
		\draw[->](4.8,3.7)--(4.8,3);
		\end{tikzpicture}
		\caption{\label{p}}
	\end{figure}
	
	\begin{prop} Given a $2$-cylinder splitting, denote the circumferences and heights of the two cylinders as $c_1$, $c_2$ and $h_1$, $h_2$, and choose the labels such that $c_1<c_2$. Then, $Area_T=\min\left\{{c_1h_1\over 2},{c_1h_2\over 2}\right\}$, where the minimum is over all possible splittings.\end{prop}
	
	\begin{proof} For each splitting shown in Figure~\ref{p1}, ${c_1h_1\over 2}$, ${c_1h_2\over 2}$ are the areas of the red and blue triangles respectively.\\
		
		\begin{figure}
			\begin{tikzpicture}[ ray/.style={decoration={markings,mark=at position .5 with {
						\arrow[>=latex]{>}}},postaction=decorate}
			]
			\draw[-](0.4,3)--(0,0.5)--(2,0.5)--(2.4,3)--(0.4,3)--(1.5,5)--(4.5,5)--(3.4,3)--(2.4,3);
			\node at (1,0.2){$c_1$};
			\node at (3,5.3){$c_2$};
			\draw[->](0.7,0.2)--(0,0.2);
			\draw[->](1.3,0.2)--(2,0.2);
			\draw[-](0,0)--(0,0.4);
			\draw[-](2,0)--(2,0.4);
			\node at (-0.3,1.7){$h_1$};
			\draw[-](-0.5,0.5)--(-0.1,0.5);
			\draw[-](-0.5,3)--(-0.1,3);
			\draw[->](-0.3,2)--(-0.3,3);
			\draw[->](-0.3,1.4)--(-0.3,0.5);
			\draw[->](2.7,5.3)--(1.5,5.3);
			\draw[->](3.3,5.3)--(4.5,5.3);
			\draw[-](1.5,5.1)--(1.5,5.5);
			\draw[-](4.5,5.1)--(4.5,5.5);
			\node at (4.8,4){$h_2$};
			\draw[-](4.6,5)--(5,5);
			\draw[-](4.6,3)--(5,3);
			\draw[->](4.8,4.3)--(4.8,5);
			\draw[->](4.8,3.7)--(4.8,3);
			\draw[fill=red](0.4,3)--(2,0.5)--(2.4,3)--cycle;
			\draw[fill=blue](0.4,3)--(2.4,3)--(1.5,5)--cycle;
			\draw[fill=green](2.4,3)--(3.4,3)--(4.5,5)--cycle;
			\draw[fill=purple](1.5,5)--(2.4,3)--(4.5,5)--cycle;
			\draw[-,very thick, yellow](2.4,3)--(3.5,5);
			\end{tikzpicture}

			\caption{\label{p1}}
		\end{figure}
		
		Given any embedded triangle formed by saddle connections, split the surface in the direction of one of the sides of this triangle, denoted as $a$. 
		Choose $a$ as the base, then the height of the triangle with regard to $a$ can not be smaller than the height of the cylinder(s) bordering $a$. 
		Hence, if the splitting is as shown in Figure~\ref{p1}, the area of this triangle can not be smaller than the minimum of the areas of the green, blue, red and purple triangles. 
		The area of the purple triangle is strictly larger than the blue triangle because it has the same height and a longer base. 
		Furthermore, after a re-splitting of the surface along the yellow line, the part to the right of the yellow line and the part to the left of the yellow line form two cylinders $E_1'$ and $E_2'$. 
		We can see that the area of the green triangle would be half of the area of $E_2'$, in other words, the green triangle has area ${c'_1h'_1\over 2}$, where $c'_i$ and $h'_i$ are the circumferences and heights of the cylinders in the new splitting.\end{proof}
	
	According to Theorem 3.3 of \cite{1}, after a $GL(2,\mathbb{R})$ action, we can make any splitting into one of the finitely many prototypes. 
	Each of these prototypes corresponds to an integer tuple $(a,b,c,e)$, and is illustrated in Figure~\ref{p2}. Here $\lambda^2=e\lambda+d$, $bc=d$, $a,b,c\in \mathbb{Z}$, $D=4d+e^2$.\\
	
	\begin{figure}
		\begin{tikzpicture}
		\draw[-](0.4,3)--(0.4,1)--(2.4,1)--(2.4,3)--(0.4,3)--(1.5,5)--(4.5,5)--(3.4,3)--(2.4,3);
		\node at (1.4,0.7) {$\lambda$};
		\draw[->](1.1,0.7)--(0.4,0.7);
		\draw[->](1.7,0.7)--(2.4,0.7);
		\draw[-](0.4,0.5)--(0.4,0.9);
		\draw[-](2.4,0.5)--(2.4,0.9);
		\node at (0.1,2){$\lambda$};
		\draw[->](0.1,1.7)--(0.1,1);
		\draw[->](0.1,2.3)--(0.1,3);
		\draw[-](-0.1,1)--(0.3,1);
		\draw[-](-0.1,3)--(0.3,3);
		\node at (1,5.3){$a$};
		\draw[-](0.4,5.1)--(0.4,5.5);
		\draw[->](0.7,5.3)--(0.4,5.3);
		\draw[->](1.3,5.3)--(1.5,5.3);
		\node at (3,5.3){$b$};
		\draw[->](2.7,5.3)--(1.5,5.3);
		\draw[->](3.3,5.3)--(4.5,5.3);
		\draw[-](1.5,5.1)--(1.5,5.5);
		\draw[-](4.5,5.1)--(4.5,5.5);
		\node at (4.8,4){$c$};
		\draw[-](4.6,5)--(5,5);
		\draw[-](4.6,3)--(5,3);
		\draw[->](4.8,4.3)--(4.8,5);
		\draw[->](4.8,3.7)--(4.8,3);
		\
		\end{tikzpicture}

		\caption{\label{p2}}
	\end{figure}
	
	Define the number $e_D$ as the greatest integer that is both smaller than $\sqrt{D}$ and congruent to $D$ mod 2. Hence $\sqrt{D}-e_D\leq 2$. So, 
	$$\min_{M\in E_D\cap \mathcal{H}(2)}Area_T(M)
	=\inf_\lambda\left({\lambda^2\over 2(d+\lambda^2)},{\lambda\over 2(d+\lambda^2)}\right)=\min\left({1\over 2\sqrt{D}}, {\sqrt{D}-e_D\over 4\sqrt{D}}\right)= {\sqrt{D}-e_D\over 4\sqrt{D}}.$$\\
	
	Furthermore, when $D\equiv 1 \mod{8}$, $(a,b,c,e)=\left(0,{D-e_D^2\over 4}, 1, -e_D\right)$ and $(a,b,c,e)=\left(0,1, {D-e_D^2\over 4}, -e_D\right)$ are prototypes of lattice surfaces that are not affinely equivalent, according to Theorem 5.3 in \cite{1}. 
	The areas of red triangles corresponding to these two prototypes dividing by the total area of these surfaces are both $\sqrt{D}-e_D\over 4\sqrt{D}$, hence surfaces belonging to both components in $E_D\cap\mathcal{H}(2)$ have the same $Area_T$. 
	Hence, $Area_T(M)= {\sqrt{D}-e_D\over 4\sqrt{D}}$ for all non-square $D$ and all $M\in E_D\cap \mathcal{H}(2)$.\\
	
	Now consider $Area_{VT}$. 
	Split the surface in the direction of one of the saddle connections as in Figure~\ref{p}, then the other saddle connection has to cross through either $E_1$ or $E_2$. 
	So, the length of their cross product has to be larger than $\min\{c_1,c_2-c_1\}\min\{h_1,h_2\}=\min\{c_1h_2,h_2(c_2-c_1),c_1h_1,h_1(c_2-c_1)\}$. 
	On the other hand, $c_1h_2$, $h_2(c_2-c_2)$, and $a_1b_1$ are twice the areas of the blue, green, and red triangles respectively, so 
	$$Area_{VT}=\min\left(Area_T,\min\left\{{a_2(b_1-b_2)\over 2Area(M)}\right\}\right).$$
	The second minimum goes through all 2-cylinder splittings, or equivalently, all splitting prototypes. 
	Hence, 
	$$\min\left\{{a_2(b_1-b_2)\over 2Area(M)}\right\}=\min_{prototype}{(b-\lambda)\lambda\over 2(d+\lambda^2)}=\min_{prototype}{2b-e-\sqrt{D}\over 4\sqrt{D}}.$$
	Because $2b-e\equiv D \mod{2}$, 
	$${2b-e-\sqrt{D}\over 4\sqrt{D}}\geq {2+e_D-\sqrt{D}\over 4\sqrt{D}}.$$
	On the other hand, the prototype $(a,b,c,e)=\left(0,1,{D-e_D^2\over 4}, -e_D\right)$ satisfies 
	$${2b-e-\sqrt{D}\over 2\sqrt{D}}={2+e_D-\sqrt{D}\over 2\sqrt{D}}.$$
	So, 
	$$\min_{M\in E_D\cap \mathcal{H}(2)}Area_{VT}(M)=\min\left(S_V(M),{2+e_D-\sqrt{D}\over 4\sqrt{D}}\right).$$
	Therefore, when $D\not\equiv 1 \mod{8}$, $Area_{VT}(M)=\min\left({\sqrt{D}-e_D\over 4\sqrt{D}},{2+e_D-\sqrt{D}\over 4\sqrt{D}}\right)$.\\
	
	When $D\equiv 1\mod{8}$, the prototypes $(a,b,c,e)=\left(0,e_D+1-{D-e_D^2\over 4},{D-e_D^2\over 4},e_D-{D-e_D^2\over 2}\right)$ and $(a,b,c,e)=\left(1,e_D+1-{D-e_D^2\over 4},{D-e_D^2\over 4},e_D-{D-e_D^2\over 2}\right)$ lie on different Teichm\"uller curves, and both of them satisfy ${2b-e-\sqrt{D}\over 2\sqrt{D}}={2+e_D-\sqrt{D}\over 2\sqrt{D}}$. 
	Hence, both components have the same $Area_{VT}$. 
	In conclusion, $Area_{VT}(M)= \min\left({\sqrt{D}-e_D\over 4\sqrt{D}},{2+e_D-\sqrt{D}\over 4\sqrt{D}}\right)$ for all $M\in E_D\cap \mathcal{H}(2)$.\end{proof}

The proofs of Theorem 1.3-1.5 are similar. 
For Theorem 1.3 the necessary prototypes of cylinder decompositions are described in section 4 of \cite{6}. 
There are only 3 types of cylinder configurations, but the Dehn twist vectors can be different, and they
are parametrized by a 5-tuple $(w, h, t, e, \epsilon)\in\mathbb{Z}^5$. 
For Theorem 1.4, the prototypes are described in \cite{ph} by $(m-1)\times(n-1)$ grid graphs and proved by \cite{wright2012schwarz}, and there are only 2 of them corresponding to the horizontal and vertical cylinders in the pair of cylinder decompositions defined by the grid graph. 
For Theorem 1.5, in either of the 3 cases, the only possible pair of cylinder decompositions has been described in \cite{leininger2004groups} by one of the three graphs $\mathcal{E}_6$, $\mathcal{E}_7$ or $\mathcal{E}_8$.

\section{Enumeration of lattice surfaces with $Area_{VT}>0.05$}
\label{sec:3}

Now we prove Theorem 1.1 using the algorithm in \cite{5}, which provides a way to list all lattice surfaces and calculate their Veech groups. 
The algorithm is based on analyzing all Thurston-Veech structures consisting of less than a given number of rectangles.\\

Let $M$ be a lattice surface. 
After an affine transformation, we can let the two saddle connections that form the smallest virtual triangle be in the horizontal and the vertical directions without loss of generality. 
The Thurston-Veech construction \cite{thurston1988geometry} gives a decomposition of $M$ into rectangles using horizontal and vertical saddle connections. The surface $M$ is, up to scaling, completely determined by the configuration of those rectangles as well as the ratios of moduli of horizontal and vertical cylinders. Hence, we can find all lattice surfaces with given $Area_{VT}$ by analyzing all possible Thurston-Veech structures.\\ 

Smillie and Weiss presented their algorithm in the following way:\\

\begin{enumerate}
	\item Fix $\epsilon>0$, find all possible pairs of cylinder decompositions with less than $\left\lfloor{1\over 2\epsilon}\right\rfloor$ rectangles (there are finitely many such pairs), calculate their intersection matrices and decide the position of cone points.
	\item For any number $k$, find all possible Dehn twist vectors for a $k$-cylinder decomposition.
	\item Use the result from Step (1) and (2) to determine the shape of all possible flat surfaces, and rule out most of them with criteria based on \cite{5}, which we will state explicitly later.
	\item Rule out the remaining surfaces by explicitly finding pairs of saddle connections with holonomy vectors $l$, $l'$ such that ${{||l\times l'||}\over 2Area(M)}<\epsilon$.\\
\end{enumerate}

In step (2) and (3), we made some modification to improve the efficiency, which we will describe below.\\ 

Now we describe these steps in greater detail.\\

Step 1: Choose $\epsilon=0.05$, find all possible pairs of cylinder decompositions with less than $\left\lfloor{1\over 2\epsilon}\right\rfloor=10$ rectangles. 
Calculate their intersection matrices and decide the position of cone points.\\

A pair of cylinder decompositions partition the surface into finitely many rectangles. 
Let $r$ be the permutation of those rectangles that send each rectangle to the one to its right, and $r'$ be the permutation that send each rectangle to the one below, then the cylinder intersection pattern can be described by these two permutations. 
According to \cite{5}, if $M$ is a lattice surface with $Area_{VT}=\epsilon$, the pair of cylinder decomposition described as above have to decompose the surface into fewer than ${1\over 2\epsilon}$ rectangles. 
Therefore, in this step, we only need to find all transitive pairs of permutations of 9 or less elements up to conjugacy. Furthermore, we do not need to consider those pairs that correspond to a surface of genus $2$ or lower, because lattice surfaces of genus 2 or lower have already been fully classified. We also disregard those with one-cylinder decomposition in either the horizontal or vertical direction, because in either case the surface is square-tiled. In order to speed up the conjugacy check of pairs of permutations, we firstly computed the conjugacy classes of all permutations of less than $10$ elements and put them in a look-up table. Then, whenever we need to check if $r_1$, $r'_1$ and $r_2$, $r'_2$ are conjugate, we can first check if $r_1$ and $r'_1$, as well as $r_2$ and $r'_2$, belong to the same conjugacy classes.\\

Next, we calculate the following data for these cylinder decomposition: (1) the intersection matrix $A$; (2) three matrices $V, H, D$, with entries either 0 or 1, defined as follows:\begin{itemize}
	\item $V(i,j)=1$ iff the $i$-th horizontal cylinder intersects with the $j$-th vertical cylinder, and in their intersection there is at least one rectangle such that its upper-left and lower-left corners, or upper-right and lower-right corners are cone points; 
	\item $H(i,j)=1$ iff the $i$-th horizontal cylinder intersects with the $j$-th vertical cylinder, and in their intersection there is at least one rectangle such that its upper-left and upper-right corners, or lower-left and lower-right corners are both cone points; 
	\item $D(i,j)=1$ iff the $i$-th horizontal cylinder intersects with the $j$-th vertical cylinder, and in their intersection there is at least one rectangle such that its lower-right one and the upper-left corners are both cone points. 
\end{itemize}
The matrices $V, H, D$ will be used in the criteria in step 3. 
To decide whether or not the lower-right corner of the $i$-th rectangle is a cone point, we calculate $r'r(i)$ and $rr'(i)$ and check if they are different.\\

Step 2: For any number $k$, find all possible Dehn twist vectors for a $k$-cylinder decomposition.\\

Equation (9) in the proof of Proposition 3.6 of \cite{5} shows that, if the ratio between the $i$-th and the $j$-th entries of a Dehn twist vector is $p/q$, where $p$, $q$ are natural numbers and $\gcd(p,q)=1$, then $pq\leq A_iA_j/\beta^2$, where $\beta$ is an upper bound of $Area_{VT}$, and $A_i$ is the area of the $i$-th cylinder divided by the total area. 
On the other hand, by Cauchy-Schwarz inequality,
\begin{align*}
\sum_{i\neq j}A_iA_j=&{1\over 2}\left(\left(\sum_i A_i\right)^2-\sum_i A_i^2\right)\\
=&{1\over 2}\left(1-\sum_iA_i^2\right)\leq {1\over 2}\left(1-{1\over k}\left(\sum_i A_i\right)^2\right)\\
=&{k-1\over 2k}.
\end{align*}

Hence, we have:
\begin{prop}The vector $(n_1,\dots, n_k)$ cannot be a Dehn twist vector for a surface with $Area_{VT}<\beta$ if $\sum_{1\leq i<j\leq k} s_{ij}\geq {k-1\over 2k\beta^2}$, where $s_{ij}=n_in_j/\gcd(n_i,n_j)^2$.
\end{prop}\qed

Step 3: Determine the shape of possible surfaces, rule out most of those with area of virtual triangle smaller than $\epsilon=0.05$.\\

For each tuple $(A, V, H, D)$ obtained in step 1, and each Dehn twist vector obtained in step 2, we can calculate the widths and circumferences of cylinders by finding Peron-Frobenious eigenvector as in \cite{thurston1988geometry}. Then, we normalize the total area to 1 and check them against the following criteria:\\

\begin{enumerate}
	\item Let $w_i$ and $w'_j$ be the widths of the $i$-th and $j$-th cylinder in the horizontal and the vertical directions, then $w_iw'_j>1/10$. 
	This follows from the proof of Proposition 5.1 in \cite{5}.
	\item Let $c_i$ and $c'_j$ be the circumference of the $i$-th and $j$-th cylinder in the horizontal and vertical direction respectively. 
	If the ratio between the modulus of the $i$-th horizontal cylinder and the $i'$-th horizontal cylinder is $p/q$, where $p,q$ are coprime integers, then $c_iw_{i'}/q>1/10$. 
	This follows from the proof of Proposition 3.5. 
	\item With the same notation as above, if there are two cone points on the boundary of $i$-th horizontal cylinder with distance $w'_j$, and $w'_j$ is not $k_0c_i/q$ for some integer $k_0$, then for any integer $k$, $\max(|w'_j-kc_j/q|,(c_i/q-|w'_j-kc_i/q|)/2)w_{i'}>1/10$. 
	This is due to an argument similar to the proof of Proposition 3.5 as follows: after a suitable parabolic affine action we can assume that there is a vertical saddle connection crossing the $i'$-th cylinder.
	Let the holonomy vectors of two saddle connections crossing the $i$-th cylinder from a same cone point to those two cone points be $(x+nc_i, w_i)$ and $(x+w'_j+n'c_i, w_i)$, where $n,n'\in\mathbb{Z}$. 
	Do a parabolic affine action on the surface that is a Dehn twist on the $i'$-th cylinder, then their holonomy vectors will become $(x+rc_i/q+nc_i, w_i)$ and $(x+w'_j+rc_i/q+n'c_j, w_i)$, where $\gcd(r,q)=1$. 
	Repeatedly doing such affine actions, we can see that the absolute value of horizontal coordinate of the holonomy vector of at least one saddle connection we get is nonzero and no larger than $\max(|w'_j-kc_j/q|,(c_i/q-|w'_j-kc_i/q|)/2)$.
	\item Criteria (2) and Criteria (3) applies to vertical, instead of horizontal cylinders.
	\item The cross product of the holonomy vectors of diagonal saddle connections from the upper-left corner to the lower-right corner must be either 0 or larger than 1/10.\\
\end{enumerate}

In our calculation, we used an optimization which rules out some Dehn twist vectors before the calculation of Peron-Frobenious eigenvector. 
Firstly, in Step 1, we label the cylinders by the number of rectangles they contain in decreasing order. Then, when we generate Dehn twist vectors, we calculate the product of the last two entries. 
Now the last two horizontal or vertical cylinders always have the least number of rectangles, and the sum of their areas is less than $(1-c/10)$ of the total area, where $c$ is the number of rectangles not in these two cylinders. Hence, we can bound the product of their areas which in turn gives an upper bound on the product of the last two elements of the Dehn twist vector. We used the C++ linear algebra library Eigen, and the first 3 steps were done in a few hours.\\

If a 4-tuple $(A, V, H, D)$ and a pair of Dehn twist vectors pass through all the above-mentioned tests, they are printed out together with the eigenvector $(w_i)$. Below is a sample of the output of this step:\\

$A=\left(\begin{array}{ccc}3&1&1\\1&0&0\end{array}\right),
V=\left(\begin{array}{ccc}1&1&1\\1&0&0\end{array}\right),
H=\left(\begin{array}{ccc}1&0&1\\1&0&0\end{array}\right),
D=\left(\begin{array}{ccc}1&0&1\\1&0&0\end{array}\right)$\\
$n=(2,7), n'=(2,5,5), w=(1,1)$\\

$A=\left(\begin{array}{cc}6&1\\1&1\end{array}\right),
V=\left(\begin{array}{cc}0&0\\0&0\end{array}\right),
H=\left(\begin{array}{cc}0&0\\0&0\end{array}\right),
D=\left(\begin{array}{cc}0&0\\0&1\end{array}\right)$\\
$n=(1,4), n'=(1,4), w=(1,1.23607)$\\
$n=(2,7), n'=(2,7), w=(1,1)$\\

$A=\left(\begin{array}{cc}5&2\\1&1\end{array}\right),
V=\left(\begin{array}{cc}0&0\\0&0\end{array}\right),
H=\left(\begin{array}{cc}1&0\\0&0\end{array}\right),
D=\left(\begin{array}{cc}1&0\\0&1\end{array}\right)$\\
$n=(2,7), n'=(1,2), w=(1,1)$\\

$A=\left(\begin{array}{cc}6&1\\2&0\end{array}\right),
V=\left(\begin{array}{cc}1&1\\0&0\end{array}\right),
H=\left(\begin{array}{cc}1&1\\1&0\end{array}\right),
D=\left(\begin{array}{cc}0&1\\0&0\end{array}\right)$\\
$n=(1,4), n'=(1,16), w=(1,1)$\\
$n=(2,7), n'=(1,8), w=(1,1)$\\

$A=\left(\begin{array}{ccc}5&1&1\\1&1&0\end{array}\right),
V=\left(\begin{array}{ccc}1&0&1\\0&0&0\end{array}\right),
H=\left(\begin{array}{ccc}1&0&1\\0&0&0\end{array}\right),
D=\left(\begin{array}{ccc}1&0&1\\0&1&0\end{array}\right)$\\
$n=(1,4), n'=(1,3,12), w=(1,1)$\\
$n=(2,7), n'=(1,3,6), w=(1,1)$\\

Each 4-tuple $(A, V, H, D)$ is followed by pairs of Dehn twist vectors $n$, $n'$ in the horizontal
and vertical directions respectively, and a vector of widths $w$. This section of the
output described 8 combinations of $(A, V, H, D)$ and Dehn twist vectors, only the second
one will result in a non-arithmetic surface, while other line all correspond to square-tiled
cases, which we verified through integer arithmetic.\\

After collecting all tuples $(A,V,H,D)$ that may generate non-arithmetic surfaces that pass the test in this step, we can use the same algorithm in step 1 to find all pairs of permutations corresponding to these tuples, hence completely decide the shape of surfaces we need to check in the next step.\\

Step 4: After the previous 3 steps, we can show that any lattice surface with area of the smallest virtual triangle larger than 1/20 is either of genus 2, or $GL(2,\mathbb{R})$-equivalent to one of the 50 remaining cases. Two of them are the Prym eigenform of discriminant 8 in genus 3. By finding saddle connections on the remaining 48 surfaces by hand, we showed that none of them has $Area_{VT}$ greater than 1/20.\\

An example of one of the 48 surfaces is shown in Figure~\ref{p4}.\\

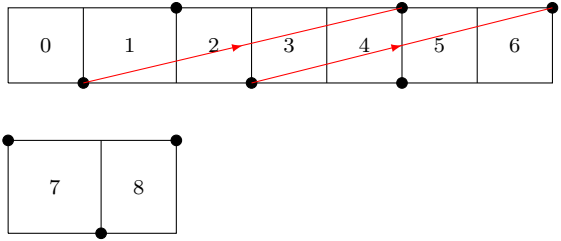
\begin{figure}
	\begin{tikzpicture}[ ray/.style={decoration={markings,mark=at position .5 with {
				\arrow[>=latex]{>}}},postaction=decorate}
	]
	\def \l {1.236};
	\draw[-](0,2)--(6+\l,2)--(6+\l,3)--(0,3)--(0,2);
	\draw[-](1,2)--(1,3);
	\draw[-](1+\l,2)--(1+\l,3);
	\draw[-](2+\l,2)--(2+\l,3);
	\draw[-](3+\l,2)--(3+\l,3);
	\draw[-](4+\l,2)--(4+\l,3);
	\draw[-](5+\l,2)--(5+\l,3);
	\node at (0.5,2.5){$0$};
	\node at (1+\l/2,2.5){$1$};
	\node at (1.5+\l,2.5){$2$};
	\node at (2.5+\l,2.5){$3$};
	\node at (3.5+\l,2.5){$4$};
	\node at (4.5+\l,2.5){$5$};
	\node at (5.5+\l,2.5){$6$};
	\draw[fill] (1,2) circle (2pt);
	\draw[fill] (1+\l,3) circle (2pt);
	\draw[fill] (2+\l,2) circle (2pt);
	\draw[fill] (4+\l,2) circle (2pt);
	\draw[fill] (4+\l,3) circle (2pt);
	\draw[fill] (6+\l,3) circle (2pt);
	\draw[fill] (0,\l) circle (2pt);
	\draw[fill] (\l,0) circle (2pt);
	\draw[fill] (1+\l,\l) circle (2pt);
	\draw[-] (\l,0)--(\l,\l);
	\draw[ray,red](1,2)--(4+\l,3);
	\draw[ray,red](2+\l,2)--(6+\l,3);
	\node at (\l/2,\l/2){$7$};
	\node at (\l+0.5,\l/2){$8$};
	\draw[-](0,0)--(\l+1,0)--(\l+1,\l)--(0,\l)--(0,0);
	\end{tikzpicture}
	\caption{\label{p4} Permutations: (in cycle notation) $(0,1,2,3,4,5,6)(7,8)$, $(0,4,6,3,5,2,8)(1,7)$; Dehn twist vectors: $(1,4)$, $(1,4)$. Dots are cone points. The holonomies of the two red saddle connections depicted have a cross product less than $1/10$ of the surface area.}
\end{figure}

Below is the list of all the 48 surfaces we checked by hand, none has $Area_{VT}>1/20$. All surfaces are represented by a pair of permutations (written in cycle notation, the $i$-th cycle is the $i$-th (horizontal or vertical) cylinder) and two Dehn twist vectors.\\

\begin{enumerate}
	\item Dehn twist vectors: (3,4), (1,2); \\
	pair of permutations: \\
	(0,1,2,3,4)(5,6,7,8), (0,3,1,6,5,8)(4,2,7)
	\item Dehn twist vectors: (1,4), (1,4); \\
	pairs of permutations:\\
	(0,1,2,3,4,5,6)(7,8), (0,3,4,5,6,2,8)(1,7)\\
	(0,1,2,3,4,5,6)(7,8), (0,4,3,5,6,2,8)(1,7)\\
	(0,1,2,3,4,5,6)(7,8), (0,5,3,4,6,2,8)(1,7)\\
	(0,1,2,3,4,5,6)(7,8), (0,4,5,3,6,2,8)(1,7)\\
	(0,1,2,3,4,5,6)(7,8), (0,3,5,4,6,2,8)(1,7)\\
	(0,1,2,3,4,5,6)(7,8), (0,5,4,3,6,2,8)(1,7)\\
	(0,1,2,3,4,5,6)(7,8), (0,6,3,4,5,2,8)(1,7)\\
	(0,1,2,3,4,5,6)(7,8), (0,4,6,3,5,2,8)(1,7)\\
	(0,1,2,3,4,5,6)(7,8), (0,5,6,3,4,2,8)(1,7)\\
	(0,1,2,3,4,5,6)(7,8), (0,4,5,6,3,2,8)(1,7)\\
	(0,1,2,3,4,5,6)(7,8), (0,6,3,5,4,2,8)(1,7)\\
	(0,1,2,3,4,5,6)(7,8), (0,5,4,6,3,2,8)(1,7)\\
	(0,1,2,3,4,5,6)(7,8), (0,3,6,4,5,2,8)(1,7)\\
	(0,1,2,3,4,5,6)(7,8), (0,6,4,3,5,2,8)(1,7)\\
	(0,1,2,3,4,5,6)(7,8), (0,5,3,6,4,2,8)(1,7)\\
	(0,1,2,3,4,5,6)(7,8), (0,6,4,5,3,2,8)(1,7)\\
	(0,1,2,3,4,5,6)(7,8), (0,3,5,6,4,2,8)(1,7)\\
	(0,1,2,3,4,5,6)(7,8), (0,5,6,4,3,2,8)(1,7)\\
	(0,1,2,3,4,5,6)(7,8), (0,3,5,6,4,2,8)(1,7)\\
	(0,1,2,3,4,5,6)(7,8), (0,3,4,6,5,2,8)(1,7)\\
	(0,1,2,3,4,5,6)(7,8), (0,4,3,6,5,2,8)(1,7)\\
	(0,1,2,3,4,5,6)(7,8), (0,6,5,3,4,2,8)(1,7)\\
	(0,1,2,3,4,5,6)(7,8), (0,4,6,5,3,2,8)(1,7)\\
	(0,1,2,3,4,5,6)(7,8), (0,3,6,5,4,2,8)(1,7)
	\item Dehn twist vectors: (1,2), (1,3); \\
	pairs of permutations: \\
	(0,1,2,3,4)(5,6,7), (0,5,1,6,2,7)(3,4)\\
	(0,1,2,3,4)(5,6,7), (1,4,6,2,3,5)(0,7)
	\item Dehn twist vectors: (1,2), (1,2); \\
	pairs of permutations: \\
	(0,1,2,3,4)(5,6,7), (0,2,4,6,7)(1,3,5)\\
	(0,1,2,3,4)(5,6,7), (0,4,3,6,7)(1,2,5)\\
	(0,1,2,3,4)(5,6,7), (0,5,2,3,7)(1,3,6)\\
	(0,1,2,3,4)(5,6,7), (1,3,6,2,5)(0,4,7)\\
	(0,1,2,3,4)(5,6,7), (0,5,1,4,7)(2,3,6)\\
	(0,1,2,3,4)(5,6,7), (0,6,2,3,7)(1,4,5)\\
	(0,1,2,3,4)(5,6,7), (0,5,1,3,7)(2,4,6)\\
	(0,1,2,3,4)(5,6,7), (0,4,6,2,7)(1,3,5)\\
	(0,1,2,3,5)(5,6,7), (0,6,3,2,7)(1,4,5)
	\item Dehn twist vectors: (1,2), (1,1); \\
	pair of permutations: \\
	(0,1,2,3,4)(5,6,7), (0,3,6,7)(1,4,2,5)
	\item Dehn twist vectors: (2,7), (2,7); \\
	pair of permutations: \\
	(0,1,2,3,4,5)(6,7), (0,3,4,5,2,7),(1,6)
	\item Dehn twist vectors: (1,1), (1,1); \\
	pair of permutations: \\
	(0,1,2,3)(4,5,6), (0,3,5,6)(1,2,4)
	\item Dehn twist vectors: (1,2), (1,2); \\
	pairs of permutations: \\
	(0,1,2,3)(4,5,6), (0,3,5,6)(1,2,4)\\
	(0,1,2,3)(4,5,6), (0,2,4,6)(1,3,5)\\
	(0,1,2,3)(4,5,6), (1,4,2,5)(0,3,6)\\
	(0,1,2,3)(4,5,6), (1,5,2,4)(0,3,6)\\
	(0,1,2,3)(4,5,6), (0,2,4,6)(1,5,3)\\
	(0,1,2,3)(4,5,6), (0,1,4,6)(2,5,3)
	\item Dehn twist vectors: (1,3), (1,3); \\
	pairs of permutations: \\
	(0,1,2,3,4)(5,6), (0,3,4,2,6)(1,5)
	\item Dehn twist vectors: (1,3), (1,3); \\
	pairs of permutations: \\
	(0,1,2,3)(4,5), (0,2,3,5)(1,4)\\
	(0,1,2,3)(4,5), (0,1,3,5)(2,4)
\end{enumerate}

And the two lattice surfaces that are Prym eigenforms in genus 3 are as follows: \begin{itemize}
	\item Dehn twist vectors (1,1,1), (1,1,1);\\
	pairs of permutations:\\ 
	(0,1,2)(3,4)(5,6), (0,4,6)(1,5)(2,3)\\
	(0,1,2)(3,4)(5,6), (1,4,5)(0,6)(2,3).
\end{itemize}

%

\begin{acknowledgements}
The author thanks his thesis advisor John Smillie for suggesting the problem and for many helpful conversations, and Alex Wright and Anja Randecker for many helpful comments.
\end{acknowledgements}

\bibliographystyle{spmpsci}      
\bibliography{template}   

%
%

\end{document}